\documentclass[12pt]{amsart}
\usepackage{epsfig,amscd,amssymb,amsxtra,amsmath,amsthm,multicol,makecell}

\newcommand{\wit}{\widetilde}

\newcommand{\calU}{{\mathcal{U}}}

\newcommand{\calC}{{\mathcal{C}}}

\numberwithin{equation}{section}

\newtheorem{Prop}[equation]{Proposition}
\newtheorem{Lem}[equation]{Lemma}
\newtheorem{Def}[equation]{Definition}
\newtheorem{Thm}[equation] {Theorem}
\newtheorem{Cor}[equation]{Corollary}
\newtheorem{Rem}[equation]{Remark}

\title
[Computation of Certain Galois Groups]
{Hilbert's Irreducibility, Modular Forms, and  Computation of Certain Galois Groups}

\author{Iva Kodrnja}
\address{
	Faculty of Geodesy, 
	University of Zagreb,
	Ka\v ci\' ceva 26, 10000 Zagreb,
	Croatia}
\email{ikodrnja@geof.hr}

\author{ Goran Mui\'c}
\address{
  Department of Mathematics,
  Faculty of Sciences,
University of Zagreb,
Bijeni\v cka 30, 10000 Zagreb,
Croatia}
 \email{gmuic@math.hr}

\begin{document}

\begin{abstract}
In this paper we discuss applications of our earlier work  in studying
certain Galois groups and splitting fields of rational functions in $\mathbb Q\left(X_0(N)\right)$ using Hilbert's irreducibility theorem and modular forms. We also consider computational aspect of the
problem using MAGMA and SAGE.
\end{abstract}


\thanks{The  second author acknowledges Croatian Science Foundation grant IP-2018-01-3628.}
\maketitle

\section{Introduction}\label{intr}

In this paper we consider congruence subgroups $\Gamma_0(N)$, $N\ge 1$, and the corresponding compact Riemann surface $X_0(N)$ which we initially consider as a complex irreducible smooth projective curve.
The $\mathbb Q$--structure on $X_0(N)$ is defined in a standard way using $j$--function i.e., the field of rational functions over $\mathbb Q$ on $X_0(N)$ is given by
$\mathbb Q\left(X_0(N)\right)= \mathbb Q(j, j(N \cdot))$.
We elaborate on equivalent ways of defining the $\mathbb Q$--structure in Section \ref{agr} (see Lemma \ref{agr-2} and its proof). This should be well--known but it is hard to find a
convenient reference. These are technical results used in the proof of the main results below.  In our earlier paper \cite{Muic2} we gave fairly general study  of complex holomorphic maps
$X_0(N)\longrightarrow \mathbb P^2$ and developed the test for  birationality in terms of modular forms of various even weights. The methods of that paper are further extended and combined with techniques 
of explicit computations in SAGE in our paper \cite{MuKo}. The paper \cite{MuKo} constructs various models over $\mathbb C$ of $X_0(N)$ complementing previous works such as \cite{BKS}, \cite{bnmjk},
\cite{sgal}, \cite{ishida}, \cite{Kodrnja1}, \cite{Muic}, \cite{MuMi}, \cite{mshi} and \cite{yy}. In the present paper we describe another application of  \cite{Muic2} and \cite{MuKo} where we justify
why it is interesting to study plane models of $X_0(N)$ of various type not necessarily of smallest possible degree or one with smallest possible coefficients.  By Eichler--Shimura theory
\cite[Theorem 3.5.2]{shi} and explicit determination of certain Eisenstein series in  \cite[Chapter 7]{Miyake}, we know that $S_m(\Gamma_0(N))$,  and
$M_m(\Gamma_0(N))$, for $m\ge 2$ even,  have basis consisting of forms with integral $q$--expansions. So, if we take that $f, g, h$ are linearly independent  modular forms
with rational $q$--expansions for $\Gamma_0(N)$, we can construct an irreducible over $\mathbb Z$ homogeneous polynomail with integral coefficients $P_{f, g, h}$ such that $P_{f, g, h}(f, g, h)=0$ in
$\mathbb Q\left(X_0(N)\right)$. Let $Q_{f, g, h}$ be its dehomogenization with respect to the last variable. Again, we obtain
an irreducible over $\mathbb Z$ polynomial with integral coefficients. By means of Hilbert's irreducibility theorem (see \cite{Serre}, summarized in  Lemma \ref{agr-4}),
we obtain a family of irreducible over $\mathbb Z$
polynomails with integral coefficients $Q_{f, g, h}(\lambda, \cdot)$, where $\lambda$ ranges over a certain thin set (see \cite[Proposition 3.3.5 and page 26, part 3]{Serre} or Lemma \ref{agr-4} here).
In this way we obtain a family of number fields determined as splitting fields of these polynomials all of them have the same Galois group $G_{f, g, h}$ which is the Galois group of the splitting field of
$Q_{f, g, h}(g/f, \cdot)$ over $\mathbb Q(g/f)$.  The goal of present paper is to study 
these objects as an  application of the theory developed in \cite{Muic2} and \cite{MuKo}. We should observe that the field  $\mathbb Q(g/f)$ is just the field of rational functions 
one variable over $\mathbb Q$, so the reader might wonder how the study of Galois groups in the present paper differ for example from  that in \cite{Serre}. The reason is that modular forms
are great source of concrete irreducible polynomials $Q_{f, g, h}(\lambda, \cdot)$ that are more restrictive than general polynomials over the field $\mathbb Q(\lambda)$ ($\lambda$ is a variable here).
The reader should consult Section \ref{cnmc} about many results for $\Gamma_0(72)$ which shows that.

We describe the content of the paper. In Section \ref{agr-hit} we explain how to use Hilbert's irreducibility (see Lemma \ref{agr-4}) and a technique of Frobenius on reducing modulo various primes
(see Theorem \ref{hit-main}) in order to compute $G_{f, g, h}$.  In Section \ref{agr},  we adapt  the results of paper \cite{MuKo} (see Lemmas \ref{agr-2}, \ref{agr-6}, and \ref{agr-8})
in order to study how the Galois group $G_{f, g, h}$ behaves when we fix $f$ and $g$ and let $h$ vary under certain assumptions on $m$ and $N$. The main result is contained in Theorems \ref{agr-1-a} and
\ref{agr-1}. They are of general nature  and they are illustrated by more concrete results in Section \ref{cnmc} for $\Gamma_0(63)$, $\Gamma_0(64)$, and especially for $\Gamma_0(72)$.
In the Section \ref{cmcg}, we improve  on Theorem \ref{agr-1} in the  case of hyperelliptic modular curves over $\mathbb Q$  (see Theorem \ref{cmcg-1}).
The section ends with an interesting example for $\Gamma_0(30)$.
In Section \ref{cnmc}, we discuss non--hyperelliptic modular curves (see Theorem \ref{sec-1}). We give several elegant computations of Galois groups $G_{f, g, h}$ based on methods and results of
Section \ref{agr-hit}. The results are contained in Proposition \ref{sec-3} (the case $\Gamma_0(63)$),  Proposition \ref{sec-4} (the case $\Gamma_0(64)$), and in Propositions \ref{sec-5},
\ref{sec-6}, and \ref{sec-7}  (the case $\Gamma_0(72)$).  The key result of that section is also Remark \ref{sec-8} related to $\Gamma_0(72)$.
The Galois groups of polynomials $\wit{Q}(\lambda, \cdot)$ (a normalized version of $Q(\lambda, \cdot)$, see Section \ref{agr-hit}) over $\mathbb Q(\lambda)$ can also be computed using MAGMA system and a
routine   {\it GaloisGroup}. It is also interesting to use new MAGMA routine {\it GeometricGaloisGroup} (see \cite{fiekerSutherland} for implementation) which give  the Galois groups of
polynomials $\wit{Q}(\lambda, \cdot)$ over   $\mathbb C(\lambda)$ and compare the two.  (For MAGMA implementations (and background) the reader should consult papers  \cite{fieker},
\cite{fiekerSutherland}, \cite{fiekerKluners}, and \cite{KrummSut}.)
They are not directly used in the paper. Nevertheless, we have made many experiments in MAGMA and SAGE related to the results of the present paper. 
In the final Section \ref{tpoal} we reprove  \cite[Lemma 2]{mshi} (needed in the proof of
Theorem \ref{cmcg-1}) using more elegant approach based on \cite{Muic2} combined with certain simple computations in SAGE.

 We would like to thank Nicole Sutherland for several useful conversations regarding the  implementation in MAGMA of routines for
 computing Galois groups of polynomials with coefficients in the function fields  over  $\mathbb Q$ and $\mathbb C$.

 \section{An Application of Hilbert's Irreducibility Theorem}\label{agr-hit}
 The goal of the present section is to set-up the notation for the rest of the paper and to prove Theorem \ref{hit-main}.

For  an even integer $m\geq 2$, let $M_m(\Gamma_0(N))$ (resp., $S_m(\Gamma_0(N))$)
be the space of (resp., cuspidal) modular forms of weight $m$ for $\Gamma_0(N)$.  Assume $\dim M_m(\Gamma_0(N))\geq 3$.
Let $f,g,h$ be three linearly independent modular forms in $M_m(\Gamma_0(N))$. Then, we define a  holomorphic map $X_0(N)\to \Bbb P^2$ uniquely determined by 
\begin{equation}\label{map}
\mathfrak{a}_z\longmapsto (f(z):g(z):h(z)), \ \ z\in \mathbb H.
\end{equation}
Since  $X_0(N)$ has a canonical structure of complex projective irreducible algebraic curve, this map can be regarded as a regular map between projective varieties. Consequently, the image
is an irreducible projective curve  which we denote by $\calC(f,g,h)$. Let $P$ be the reduced equation of $\calC(f,g,h)$ i.e., an irreducible homogeneous polynomial in three variables with complex
coefficients such that $P(f(z), g(z), h(z))=0$,  for all $z\in \mathbb H$.  It is unique up to multiplication with non--zero complex numbers. Its dehomogenization
$Q\overset{def}{=} P(1, \cdot, \cdot) $ is also an irreducible polynomial over $\mathbb C$, and it satisfies
$Q\left(g/f, h/f\right)=0$ in the field of rational functions $\mathbb C\left(X_0(N)\right)$. It is very easy to see that $P$ is homogenization of $Q$ i.e., $P$ is uniquely determined. 
The polynomial $Q$ must depend on both variables or otherwise linear independence of $f$, $g$, and $h$ is violated. We see also that
$Q\left(g/f, . \right)$ is an irreducible polynomial of $h/f\in \mathbb C\left(X_0(N)\right)$ over the subfield
$\mathbb C\left(g/f\right)$ generated by $g/f$. Consequently, the field extensions satisfy 
$\left[\mathbb C\left(g/f, \  h/f\right) : \mathbb C\left(g/f\right)\right]= \deg{\left(Q\left(g/f, \cdot \right)\right)}$.

The next  lemma is important  for our considerations. We recall that  by Eichler--Shimura theory
\cite[Theorem 3.5.2]{shi} and explicit determination of certain Eisenstein series in  \cite[Chapter 7]{Miyake}, we know that $S_m(\Gamma_0(N))$,  and
$M_m(\Gamma_0(N))$, for $m\ge 2$ even,  have basis consisting of forms with integral $q$--expansions.

\begin{Lem} \label{agr-8} Let $f,g,h$ be three linearly independent modular forms in $M_m(\Gamma_0(N))$  having rational $q$--expansions. 
  Then, we may select that the polynomial $P$ (and consequently $Q$) has rational coefficients. Moreover, $Q\left(g/f, \cdot \right)$ is an irreducible polynomial of
  $h/f$ over  $\mathbb Q\left(g/f\right)$. Consequently, we have
  $\left[\mathbb Q\left(g/f, \  h/f\right) : \mathbb Q\left(g/f\right)\right]=\left[\mathbb C\left(g/f, \  h/f\right) : \mathbb C\left(g/f\right)\right]= \deg{\left(Q\left(g/f, \cdot \right)\right)}$.
\end{Lem}
\begin{proof} The proof of the first claim of the lemma is contained in the proof of \cite[Corollary 1. 5 (ii)]{MuKo}. The polynomial $P$ must be irreducible over $\mathbb Q$ since it is irreducible
  over $\mathbb C$. The second claim follows from this. 
  \end{proof}

\begin{Def}\label{agr-9} Maintaining the assumptions of Lemma \ref{agr-8}, we denote by $Q_h=Q_{f, g, h}$ an irreducible   polynomial with relatively prime coefficients in $\mathbb Z$ 
 such that   $Q_{f, g, h}\left(g/f, h/f\right)=0$. Let $L_h=L_{f, g, h}$ be the splitting field $Q_h(g/f, \cdot)$ containing $\mathbb Q(g/f, h/f)$, and
 $G_h=G_{f, g, h}\overset{def}{=} Gal(L_h/\mathbb Q(g/f))$. For $\lambda\in \mathbb Z$, we let $L_{h, \lambda}=L_{f, g, h, \lambda}$ be the splitting field $Q_h(\lambda, \cdot)$, and
 $G_{h, \lambda}=G_{f, g, h, \lambda}\overset{def}{=} Gal(L_{h, \lambda}/\mathbb Q)$.
  \end{Def}
The polynomial $Q_{f, g, h}$ is determined up to a multiplication by $\pm 1$. The corresponding $P$  is denoted by $P_h=P_{f, g, h}$.   Next, we can combine the above considerations with
the celebrated Hilbert's irreducibility theorem.  Following Serre \cite{Serre}, we call  a subset $A\subset \mathbb Z$ thin if  the number of elements in the intersection of $A$ with
  a segment $[-n, n]$ is $O(n^{1/2})$ as $n\longrightarrow \infty$. 

\begin{Lem}\label{agr-4}
  There exists a thin subset $A_{f, g, h}\subset \mathbb Z$ such that $G_{f,g,h}$ is isomorphic to $G_{f, g, h, \lambda}$, 
  for $\lambda\in \mathbb Z-A_{f, g, h}$.
  \end{Lem}
\begin{proof} Since $g/f$ is transcendental over $\mathbb Q$ in sense of field theory, the field $\mathbb Q\left(\frac{g}{f}\right)$ is just a field of rational functions in one variable over $\mathbb Q$.
Now, the claim follows from the Hilbert's irreducibility (see for example \cite[Proposition 3.3.5 and page 26, part 3]{Serre}).
\end{proof}

We write $Q_{f, g, h}(\lambda, \cdot)$ in the form
$$
Q_{f, g, h}(\lambda, T)=\sum_{i=1}^n a_i(\lambda) T^i \in \mathbb Z[\lambda, T].
$$
We remark that $\lambda$ and $T$ are variables here.  By Definition \ref{agr-9}, the polynomial is irreducible over $\mathbb Z$ (not just over $\mathbb Q$). For the present purpose this form of
the polynomial is not good. We use standard trick 
$$
\wit{Q}_{f, g, h}(\lambda, T)\overset{def}{=} a_n(\lambda)^{n-1} Q(\lambda, T/a_n(\lambda))= T^n+ a_{n-1}(\lambda) T^{n-1}+ \sum_{i=1}^{n-2}  a_n(\lambda)^{n-1-i} a_i(\lambda)  T^i \in \mathbb Z[\lambda, T].
$$
This polynomial is also irreducible over $\mathbb Z$ (and over $\mathbb Q$) as it is easy to check. Recall that in the present notation we have 
$\sum_{i=1}^n a_i(g/f) (h/f)^i= Q_{f, g, h}(g/f, h/f)=0$ in $\mathbb Q(g/f, h/f)$. We fix an algebraic closure $\overline{\mathbb Q(g/f)}$ containing  $\mathbb Q(g/f, h/f)$. 
By Definition \ref{agr-9},  we may regard $G_{f, g, h}$ as the Galois group of the
splitting field $L_{f, g, h}$ inside $\overline{\mathbb Q(g/f)}$. It is obvious that $L_{f, g, h}$  is the splitting field of 
$Q_{f, g, h}(g/f, T)$ in $\overline{\mathbb Q(g/f)}$. We decompose into linear factors in  $L_{f, g, h}$
$$
\wit{Q}_{f, g, h}(g/f, T)=(T-\alpha_1)(T-\alpha_2)\cdots (T-\alpha_n), \ \ \text{say,} \  \alpha_1= a_n(g/f)h/f.
$$
The group $G_{f, g, h}$ permutes the fixed sequence $\alpha_1, \cdots, \alpha_n$. In this way, we may regard $G_{f, g, h}$ as a subgroup of the symmetric group $S(n)$ of $n$--letters.

Let $\Lambda_{f, g, h}$ be the integral closure of $\mathbb Q[\lambda]$ in $L_{f, g, h}$. 
Let  $\lambda \in \mathbb Z- A_{f, g, h}$. We select a point $\mathfrak a\in X_0(N)$ such that $\mathfrak a$ gets mapped by $g/f$ onto $\lambda$. 
Let $\mathfrak m_{\lambda}$  be the maximal ideal in $\mathbb Q[g/f]$ consisting of all that  map $\mathfrak a$ onto zero.
We fix a maximal ideal  $\mathfrak M_{f, g, h, \lambda}$ in $\Lambda_{f, g, h}$ lying over $\mathfrak m_{\lambda}$. Then, $\Lambda_{f, g, h}/\mathfrak M_{f, g, h, \lambda}$ is a splitting field
of $\wit{Q}_{f, g, h}(\lambda, T)$, and we have
$$
\wit{Q}_{f, g, h}(\lambda, T)=(T-\overline{\alpha}_1)(T-\overline{\alpha}_2)\cdots (T-\overline{\alpha}_n),
$$
where  $\overline{\alpha}_i=\alpha_i+ \mathfrak M_{f, g, h, \lambda}$. We may let $L_{f, g, h, \lambda} =\Lambda_{f, g, h}/\mathfrak M_{f, g, h, \lambda}$ since
both   $\wit{Q}_{f, g, h}(\lambda, T)$ and $\wit{Q}_{f, g, h}(\lambda, T)$ have the same splitting field. Thus, 
$G_{f, g, h, \lambda}$ permutes the fixed sequence $\overline{\alpha}_1, \cdots, \overline{\alpha}_n$ (see Lemma \ref{agr-9}). In this way, we may regard $G_{f, g, h, \lambda}$ as a transitive  subgroup of the
symmetric group $S(n)$ of $n$--letters.  The isomorphism $G_{f, g, h}$ with $G_{f, g, h, \lambda}$ (see Lemma \ref{agr-4}) can be taken to be
compatible with the action on above fixed sequences of roots.  After all preparations, classical Frobenius theory  (generalized later by Chebotarev) easily implies the following theorem which we use
for explicit computations.

\begin{Thm}\label{hit-main} Maintaining above assumptions, $G_{f, g, h}$ contains a permutation with a cycle pattern  $n_1, n_2, \ldots, n_t$ if and only if there exists a prime number $p$  and 
  and $r\in \{0, 1, \ldots, p-1\}$ such that 
  $\wit{Q}_{f, g, h}(r, T)=  T^n+ a_{n-1}(r) T^{n-1}+ \sum_{i=1}^{n-2}  a_n(r)^{n-1-i} a_i(r) T^i  \ (\text{mod} \ p)$
  can be decomposed into  a product of different irreducible factors of degrees $n_1, n_2, \ldots, n_t$.
\end{Thm}
\begin{proof} First, for each $r\in \{0, 1, \ldots, p-1\}$, the number of integers $\le n$ of the the form $kp+r$ is  $O(n)$ is $n\longrightarrow \infty$. So, this set can not be thin.  
Next, let    $\lambda \in \mathbb Z-A_{f, g, h}$.  By Frobenius theorem, the density of all primes $p$ such that  $\wit{Q}_{f, g, h}(\lambda, T) \ (\text{mod} \ p)$
  can be decomposed into  a product of different irreducible factors of degrees $n_1, n_2, \ldots, n_t$ is equal to the number number of elements in
  $G_{f, g, h, \lambda}$  with a cycle pattern  $n_1, n_2, \ldots, n_t$ divided by the number of elements in $G_{f, g, h, \lambda}$.

After these preliminaries,  we prove the theorem. First, assume  that for $r\in \{0, 1, \ldots, p-1\}$  
$\wit{Q}_{f, g, h}(r, T)$   can be decomposed into  a product of different irreducible factors of degrees $n_1, n_2, \ldots, n_t$. Then, we select  $\lambda \in \mathbb Z-A_{f, g, h}$ such that
$\lambda \equiv r \ (\text{mod} \ p)$. By Frobenius theorem and Lemma \ref{agr-4}, we obtain that $G_{f, g, h}$ contains a permutation with a cycle pattern  $n_1, n_2, \ldots, n_t$. 

Conversely, assume that  $G_{f, g, h}$ contains a permutation with a cycle pattern  $n_1, n_2, \ldots, n_t$. Then, the same is true for  $G_{f, g, h, \lambda}$  for any $\lambda \in \mathbb Z-A_{f, g, h}$.
We fix any such $\lambda$. By Frobenius theorem, there exists a prime number $p$ such that  $\wit{Q}_{f, g, h}(\lambda, T) \ (\text{mod} \ p)$
  can be decomposed into  a product of different irreducible factors of degrees $n_1, n_2, \ldots, n_t$. Finally, let  $r$ be the residue of division of $\lambda$ by $p$. We have 
$\wit{Q}_{f, g, h}(\lambda, T) \equiv  \wit{Q}_{f, g, h}(r, T)   \ (\text{mod} \ p)$
proving the claim. \end{proof} 

\section{Proof of Theorems \ref{agr-1} and  \ref{agr-1-a}}\label{agr}
The goal of the present section is to prove Theorems \ref{agr-1-a} and  \ref{agr-1}. We start by preliminary definitions and results.
Let 
$$
E_4(z)=1+240 \sum_{n=1}^\infty \sigma_3(n)q^n$$
 be the usual Eisenstein series, 
and let 
$$
\Delta(z)=q+\sum_{n=2}^\infty\tau(n)q^n
$$
 be the Ramanujan delta function. Then,   $j=E^3_4/\Delta$. 
Since  $E^3_4(N\cdot ), \Delta (N\cdot ) \in M_{12}(\Gamma_0(N))$,
the fact that the function field over $\mathbb C$ is generated by  $j$ and 
$j(N\cdot )$ means that the holomorphic map

\begin{equation}\label{mmap}
\mathfrak a_z\mapsto (1: j(z): j(Nz))=
(\Delta(z)\Delta(Nz): E^3_4(z)\Delta(Nz): E^3_4(Nz) \Delta(z))
\end{equation}
is a birational uniformization by $S_{24}(\Gamma_0(N))$. The $\mathbb Q$--structure on the curve $X_0(N)$ is introduced by considering its
field of rational function defined over $\mathbb Q$:

\begin{equation} \label{agr-100}
\mathbb Q\left(X_0(N)\right)\overset{def}{=} \mathbb Q(j, j(N \cdot)).
\end{equation}

Now, we are ready to prove the following lemma. The lemma must be well--known but we could not find appropriate reference.

\begin{Lem} \label{agr-2} Assume that  $m\ge 2$ is an even integer. Suppose that  $f, g, h\in M_m(\Gamma_0(N))$ have integral (or rational) $q$--expansions. Then, the map (\ref{map}) is birational
  over $\mathbb C$ if and only if it is birational over $\mathbb Q$. 
\end{Lem}
\begin{proof}  Assume that the map (\ref{map}) is birational
  over $\mathbb C$. Then, we have
  $$
   \mathbb C(g/f, \ h/f)= \mathbb C(j, j(N \cdot)).
 $$
This implies that there exist homogeneous polynomials $P$ and $Q$ of the same degree with complex coefficients such that
   $$
   j=\frac{P(f, g, h)}{Q(f, g, h)}.
   $$
Equivalently, we have the following:
   $$
   E^3_4 \cdot Q(f, g, h)- \Delta\cdot P(f, g, h) =0.
   $$
 Inserting rational $q$--expansions of $ E^3_4$, $\Delta$, $f$, $g$, and $h$,  we obtain a $q$--expansion of the modular form $E^3_4 \cdot Q(f, g, h)- \Delta\cdot P(f, g, h)$.
 But,  by the assumption, this modular form is identically equal to zero. Therefore, all its coefficients in the $q$--expansion are equal to zero. (Enough to choose first few depending on its weight.)
 Each of these coefficients is a linear homogeneous function in coefficients of $P$ and $Q$. Thus, we obtain  the  
 homogeneous linear system in coefficients of $Q$ and $P$ with coefficients in $\mathbb Q$. This homogeneous system has a non--zero solution over $\mathbb C$.
 Therefore, it has a non--zero solution over $\mathbb Q$. Thus, we can take that the coefficients of $P$ and $Q$ are in $\mathbb Q$. This proves
   $j\in \mathbb Q(g/f, \ h/f)$. Similarly, we also prove 
   $j(N\cdot )\in \mathbb Q(g/f, \ h/f)$.
   This implies
\begin{equation} \label{agr-3}
\mathbb Q\left(X_0(N)\right)= \mathbb Q(j, j(N \cdot))\subset \mathbb Q(g/f, \ h/f).
\end{equation}
Conversely, we can write
$$
j=\frac{E^3_4}{\Delta}= \frac{E^3_4 \cdot \Delta(N \cdot)}{\Delta \cdot \Delta(N \cdot)}, \ \text{and} \
j(N \cdot)=\frac{E^3_4(N \cdot)}{\Delta(N \cdot)}= \frac{E^3_4(N \cdot) \cdot \Delta}{\Delta \cdot \Delta(N \cdot)}.
$$
Now, we can argue as before to prove
$$
g/f, \ h/f \in \mathbb Q\left(\frac{E^3_4 \cdot \Delta(N \cdot)}{\Delta \cdot \Delta(N \cdot)}, \ \frac{E^3_4(N \cdot)\cdot \Delta}{\Delta\cdot \Delta(N \cdot)}\right)= 
\mathbb Q(j, j(N \cdot)).
$$
Hence, we obtain
$$
\mathbb Q(g/f, \ h/f)\subset \mathbb Q\left(X_0(N)\right).
$$
This means that  the map (\ref{map}) is birational
over $\mathbb Q$.

 Conversely,   we have
  $$
  \mathbb Q(g/f, \ h/f)= \mathbb Q\left(X_0(N)\right).
  $$
  Taking the composite fields (with $\mathbb C$) we obtain
   $$
  \mathbb C(g/f, \ h/f)= \left(\mathbb C\right) \left(\mathbb Q(g/f, \ h/f)\right)= \left(\mathbb C\right) \left(\mathbb Q\left(X_0(N)\right)\right)=
  \mathbb C\left(X_0(N)\right).
  $$
  In the last equality we use (\ref{agr-100}).    This means that the map (\ref{map}) is birational over $\mathbb C$.
 \end{proof}

By Eichler--Shimura theory \cite[Theorem 3.5.2]{shi}, for 
each even integer $m\ge 2$ the space of cusp forms $S_m(\Gamma_0(N))$ is defined over $\mathbb Q$. Let $S_m(\Gamma_0(N))_{\mathbb Q}$ be the $\mathbb Q$--span of all cusp forms in $S_m(\Gamma_0(N))$ 
having rational $q$--expansions. Obviously, $S_m(\Gamma_0(N))=S_m(\Gamma_0(N))_{\mathbb Q}\otimes_{\mathbb Q} \ \mathbb C$.

\begin{Lem}\label{agr-6} Assume that either $m=2$ and $X_0(N)$ is not hyperelliptic (implies $g(\Gamma_0(N))\ge 3$) or $m\ge 4$ is  an even integer such that $\dim S_m(\Gamma_0(N)) \ge \max{(g(\Gamma_0(N))+2, 3)}$.
  Then, we have the following:
\begin{itemize}
    \item[(i)] Let $f_0, \ldots, f_{s-1}$ be a basis of $S_m(\Gamma_0(N))_{\mathbb Q}$. Then, $\mathbb Q(X_0(N))$ is generated over $\mathbb Q$ by the
      quotients $f_i/f_0$, $1\le i\le s-1$.
      \item[(ii)] Assume that $f, g \in S_m(\Gamma_0(N))_{\mathbb Q}$ are linearly independent over $\mathbb Q$. 
	Then, there exists  a non-empty Zariski open set $\calU_{f, g}\subset S_m(\Gamma_0(N))_{\mathbb Q}$ \footnote{By definition, the intersection of a Zariski dense subset of
        $S_m(\Gamma_0(N))$ defined over $\mathbb Q$ with  $S_m(\Gamma_0(N))_{\mathbb Q}$.} such that $X_0(N)$ is birationally equivalent over $\mathbb Q$ to  $\calC (f, g, h)$ via the map (\ref{map}) i.e.,
        $\mathbb Q\left(g/f, h/f\right)=\mathbb Q\left(X_0(N)\right)$ for any 	$h\in \calU_{f, g}$.
      \item[(iii)] Assume that $f, g \in S_m(\Gamma_0(N))_{\mathbb Q}$ are linearly independent over $\mathbb Q$. Then, the degree of the field extension $[\mathbb Q(X_0(N)): \mathbb Q(g/f)]$
        is equal to the degree of the  divisor of zeores of $g/f$. 
       \end{itemize}
  \end{Lem}
\begin{proof}  The proof of (i) is analogous to that of  Lemma \ref{agr-2}. The detail are left to the reader. We just note that under assumption of the lemma we have
  $\mathbb C\left(f_1/f_0, \ldots, f_{s-1}/f_0\right)=   \mathbb C\left(X_0(N)\right)$.   This is obvious when $m=2$, and we use  \cite[Corollary 3.4]{Muic1} for the case $m\ge 4$.
  For the proof of the first claim in (ii), one can  adjust the proofs of \cite[Theorem 1.4]{MuKo}. We sketch the argument.  First, we select a basis of $S_m(\Gamma_0(N))_{\mathbb Q}$ such that
  $f=f_0$ and $g=f_1$.  Then, 
  by (i), we have that
  $$
  \mathbb Q\left(X_0(N)\right)=   \mathbb Q\left(g/f, f_2/f, \ldots, f_{s-1}/f\right)=  \mathbb Q\left(g/f\right) \left(f_2/f, \ldots, f_{s-1}/f_0\right) \supset  \mathbb Q\left(g/f\right)
  $$
  is a finite algebraic extension. The set of all $(\lambda_2, \ldots, \lambda_{s-1})\in \mathbb Q^{s-2}$  such that  $(\lambda_2 f_2+\cdots + \lambda_{s-1} f_{s-1})/f$ generates
  $\mathbb Q\left(X_0(N)\right)$ over $ \mathbb Q\left(g/f\right)$ is not empty (by an argument of \cite[Lemma 3.2]{MuKo}) and it is determined by  $R(\lambda_2, \ldots, \lambda_{s-1})\neq 0$.
  Here, $R$ is the discriminant of the characteristic polynomial of the endomorphism of the vector space  $\mathbb Q\left(X_0(N)\right)$ over $ \mathbb Q\left(g/f\right)$ determined by the multiplication by
  $(\lambda_2 f_2+\cdots + \lambda_{s-1} f_{s-1})/f$ (exactly as in the proof of \cite[Lemma 3.4]{MuKo}). The discriminant $R$ is a polynomial in variables $\lambda_2, \ldots, \lambda_{s-1}$ with coefficients in
  $\mathbb Q\left(g/f\right)$. By above considerations, $R$ is a non--zero polynomial. We select $\lambda\in \mathbb Q$ and $\mathfrak a\in X_0(N)$ such that $g/f$ maps $\mathfrak a$ onto $\lambda$, and
  all non--zero coefficients of $R$ when evaluated at $\mathfrak a$,  are non--zero numbers in $\mathbb Q$. Then, $R(\lambda_2, \ldots, \lambda_{s-1})(\mathfrak a)$ is a non--zero polynomial with
  rational coefficients. Then, for the Zariski open set $\calU_{f, g}\subset S_m(\Gamma_0(N))_{\mathbb Q}$ can be taken the set of all $h=\lambda_0 f+ \lambda_1 g+\lambda_2 f_2+\cdots + \lambda_{s-1} f_{s-1}$,
  such that $R(\lambda_2, \ldots, \lambda_{s-1})(\mathfrak a)\neq 0$.

  Now, we prove (iii). First of all, by (ii), there exists $S_m(\Gamma_0(N))_{\mathbb Q}$  such that $\mathbb Q\left(g/f, h/f\right)=\mathbb Q\left(X_0(N)\right)$. By Lemma \ref{agr-2}, we have
  $\mathbb C\left(g/f, h/f\right)=\mathbb C\left(X_0(N)\right)$. Now, Lemma \ref{agr-8}, we have that $[\mathbb Q(X_0(N)): \mathbb Q(g/f)]$ is equal to $[\mathbb C(X_0(N)): \mathbb C(g/f)]$.
  Finally, it is a well--known that $[\mathbb C(X_0(N)): \mathbb C(g/f)]$ is  equal to the degree of the  divisor of zeroes of $g/f$. 
  \end{proof}

We generalize the last claim in (iii) in the following
corollary:

\begin{Cor} \label{agr-16} Let $m\ge 2$ be an even integer such that $S_m(\Gamma_0(N))_{\mathbb Q}\ge 2$.
  Assume that $f, g \in S_m(\Gamma_0(N))_{\mathbb Q}$ are linearly independent over $\mathbb Q$. Then, the degree of the field extension $[\mathbb Q(X_0(N)): \mathbb Q(g/f)]$
  is equal to the degree of the  divisor of zeores of $g/f$. Moreover, if we let \begin{equation} \label{agr-5}
           l_{m, N}\overset{def}{=}\dim S_m(\Gamma_0(N))+ g(\Gamma_0(N))-1 -\epsilon_{m}, \ \ \text{$\epsilon_2=1$  and $\epsilon_m=0$ for $m\ge 4$,}
         \end{equation}
then  $[\mathbb Q(X_0(N)): \mathbb Q(g/f)]\le l_{m, N}$ (a bound independent of $f$ and $g$).
\end{Cor}
\begin{proof} Select sufficiently large $k$ such that $\dim S_{km}(\Gamma_0)\ge \max{(g(\Gamma_0(N))+2, 3)}$ (which is easy to see using explicit formula for the dimension (see for
  example  \cite[Lemma 2.2]{MuKo})). For such $k$, $f^k$ and $f^{k-1}g$ are $\mathbb Q$--linearly independent cuspidal forms in $\dim S_{km}(\Gamma_0)$. Now, the first claim of the corollary follows from
  Lemma \ref{agr-6} (iii) since  $f^{k-1}g/f^k= g/f$.

  For the bound, we estimate the degree of the divisor of zeroes of $g/f$ applying \cite[Lemma 2.1]{Muic2}. We obtain
  \begin{equation} \label{agr-10}
  \mathrm{div}{(g/f)}=\mathrm{div}{(g)}-\mathrm{div}{(f)}=\mathfrak c_g-\mathfrak c_f, 
  \end{equation} 
  where for non--zero $h_1\in S_m(\Gamma_0(N))$ the divisor $c_{h_1}$ is an  effective integral divisors  of the degree  $\deg{(\mathfrak c_{h_1})}=l_{m, N}$.
 Hence, (\ref{agr-10}) implies  that the degree of the divisor of zeroes of $g/f$ is $\le l_{m, N}$.  
 \end{proof}

\vskip .2in
Let us fix an even integer  $m\ge 2$ such that $\dim{S_m(\Gamma_0(N))_{\mathbb Q}}\ge 3$. Assume that $f, g\in S_m(\Gamma_0(N))_{\mathbb Q}$ are $\mathbb Q$--linearly independent.
Then, for each $h \in S_m(\Gamma_0(N))_{\mathbb Q}$, $h\not\in \mathbb Q f+ \mathbb Q g$, we may construct irreducible polynomials $P_h$ and $Q_h$ with integral coefficients, define splitting fields $L_h$,
and define Galois groups $G_h$ (see Definition \ref{agr-9}).

\begin{Def} \label{agr-12}
Let $\mathcal G=\mathcal G_{f, g}$ be the set consisting of all representatives of groups $G_h$, $h\in S_m(\Gamma_0(N))_{\mathbb Q} - \left(\mathbb Q f+ \mathbb Q g\right)$
up to isomorphism. For $G\in \mathcal G$,  let  $\Xi_G$ be the set of all $h\in S_m(\Gamma_0(N))_{\mathbb Q} - \left(\mathbb Q f+ \mathbb Q g\right)$ such that $G_h\simeq G$.
We denote by $\Xi'_G$ the set of all $h\in \Xi_G$ such that  the degree of $Q_{f, g, h}(g/f, \cdot)$ is $[\mathbb Q(X_0(N)): \mathbb Q(g/f)]$.
\end{Def}

We explain Definition \ref{agr-12} in the following lemma: 

\begin{Lem}\label{agr-13} Let $G\in \mathcal G$. Then, the order of $G$ is divisible by $\left[\mathbb Q(g/f, h/f): \mathbb Q(g/f)\right]$ for any
  $h\in \Xi_G$. Moreover,  $G$ can be regarded as a subgroup of the symmetric group of
  $\left[\mathbb Q(X_0(N)): \mathbb Q(g/f)\right]$--letters (non--uniquely). In particular, $\mathcal G$ is finite. Moreover,  assume that $\Xi'_G\neq \emptyset$. Then,
  the order of $G$ is divisible by $[\mathbb Q(X_0(N)): \mathbb Q(g/f)]$, $\mathbb Q(g/f, h/f)=\mathbb Q(X_0(N))$,
  the splitting field $L_{f, g, h}$ is a finite extension of
    $\mathbb Q(X_0(N))$, and  $G$ can be regarded as a transitive subgroup of the symmetric group of
$\left[\mathbb Q(X_0(N)): \mathbb Q(g/f)\right]$--letters (uniquely up to conjugation).
\end{Lem}
\begin{proof} The first claim is obvious since we have the chain of field extensions (see Definition \ref{agr-9})
  $$
  \mathbb Q(g/f)\subset \mathbb Q(g/f, h/f) \subset L_h,
  $$
  and  $G\simeq G_h=Gal(L_h/\mathbb Q(g/f))$. 
Basic elementary estimate for  $\left[L_h: \mathbb Q(g/f)\right]$ implies 
$$
\# G_h= \left[L_h : \mathbb Q(g/f)\right]\le \left(\left[\mathbb Q(g/f, h/f): \mathbb Q(g/f)\right] \right)!
\le \left(\left[\mathbb Q(X_0(N): \mathbb Q(g/f)\right]\right)!.
$$
Hence, $G$ can be regarded as a subgroup of the symmetric group of
  $\left[\mathbb Q(X_0(N)): \mathbb Q(g/f)\right]$--letters. Thus, $\mathcal G$ is finite.

Assume that $\Xi'_G\neq \emptyset$. Then, by Lemma \ref{agr-8} and Definition \ref{agr-12}, for $h\in \Xi'_g$, we obtain
  $$\left[\mathbb Q\left(g/f, \  h/f\right) : \mathbb Q\left(g/f\right)\right]= \deg{\left(Q_{f, g, h}\left(g/f, \cdot \right)\right)}=[\mathbb Q(X_0(N)): \mathbb Q(g/f)].
  $$
  This is equivalent with $\mathbb Q(g/f, h/f)=\mathbb Q(X_0(N))$. Now, it is obvious that $L_{f, g, h}$ is a finite extension of
    $\mathbb Q(X_0(N))$. Other two claims in the last part of the lemma follow from Galois theory.
\end{proof}

In the next lemma we give a bound independent of choice of $f$ and $g$. The bound might not be optimal but it is needed in the theorem below. 

\begin{Lem}\label{agr-14} The order of every group $G\in \mathcal G$ is less than or equal to $l_{m, N}!=1\cdot 2\cdots l_{m, N}$ where $l_{m, N}$ is defined by (\ref{agr-5}). 
 \end{Lem}
\begin{proof}  We argue as in Lemma \ref{agr-13} using the estimate  $\left[\mathbb Q(X_0(N): \mathbb Q(g/f)\right]\le l_{m, N}$ (see Corollary \ref{agr-16}).
\end{proof}

Finally, we need the following lemma:  

\begin{Lem}\label{agr-15}  Assume that $f, g\in S_m(\Gamma_0(N))_{\mathbb Q}$ are $\mathbb Q$--linearly independent. Then, there exists a subgroup  $G\in  \mathcal G$ such that $\Xi_G$ is Zariski dense
  in $S_m(\Gamma_0(N))_{\mathbb Q}$.
\end{Lem}
\begin{proof} The Zariski open set $ S_m(\Gamma_0(N))_{\mathbb Q} - \left(\mathbb Q f+ \mathbb Q g\right)$ is a union of finite number of sets $\Xi_G$ 
  parameterized by subgroups $G\in  \mathcal G$.  Taking the Zariski closure we find that
  $$
  S_m(\Gamma_0(N))_{\mathbb Q} = \cup_G \ \overline{\Xi}_G.
  $$
  But the space $S_m(\Gamma_0(N))_{\mathbb Q}$ is irreducible in sense of algebraic geometry. So, there exists a subgroup  $G\in \mathcal G$ such that 
  $$
  \overline{\Xi}_G=  S_m(\Gamma_0(N))_{\mathbb Q}.
  $$
  This means that $\Xi_G$ is Zariski dense in $S_m(\Gamma_0(N))_{\mathbb Q}$. 
\end{proof}

Now, after all of these preparations, we are ready to state and prove the main results of the present section.

\begin{Thm}\label{agr-1-a}   Let $m\ge 2$ be an even integer such that $\dim{S_m(\Gamma_0(N))_{\mathbb Q}}\ge 3$. 
  Then, there exists a thin subset $A_{m, N}\subset \mathbb Z$, and  triples of linearly independent forms $f_i, g_i, h_i \in S_m(\Gamma_0(N))$, $1\le i\le k$,  such that  for any 
    $f, g, h\in S_m(\Gamma_0(N))$ which are linearly independent, there exists $i$ such that
    $G_{f, g, h}\simeq G_{f_i, g_i, h_i, \lambda}$, $\lambda \in \mathbb Z-  A_{m, N}$.    
      \end{Thm}
\begin{proof} By Lemma \ref{agr-14}, when $f, g$ ranges over pairs of linearly independent forms in  $S_m(\Gamma_0(N))_{\mathbb Q}$, the groups in $\mathcal G_{f, g}$ are isomorphic to the
  subgroups of the symmetric group of $l_{m, N}$--letters. Thus, in the corresponding union $\cup_{f, g} \ \mathcal G_{f, g}$, there are only finitely many groups up to isomorphism. Let us call the
  representatives $G_{f_i, g_i, h_i}$, $1\le i\le k$.    Next,  we select thin subsets $A_{f_i, g_i, h_i}$ as in Lemma \ref{agr-4}. Then, the union
  $$
  A_{m, N}\overset{def}{=} \cup^k_{i=1} \ A_{f_i, g_i, h_i}
  $$
  is a thin subset (see the paragraph before the statement of Lemma \ref{agr-4}). For $\lambda \in \mathbb Z-  A_{m, N}$, we have   $G_{f_i, g_i, h_i,} \simeq G_{f_i, g_i, h_i, \lambda}$, $1\le i \le k$.
  The theorem follows.  \end{proof}

\begin{Thm}\label{agr-1}  Assume that either $m=2$ and $X_0(N)$ is not hyperelliptic (implies $g(\Gamma_0(N))\ge 3$) or $m\ge 4$ is  an even integer such that
  $\dim S_m(\Gamma_0(N)) \ge \max{(g(\Gamma_0(N))+2, 3)}$. Assume that $f, g \in S_m(\Gamma_0(N))$ are linearly independent. Then, 
   there exists a subgroup   $G$ of the symmetric group of $l_{m, N}$--letters (see (\ref{agr-5}))  such that  $\Xi'_G$ is Zariski dense in
    $S_m(\Gamma_0(N))_{\mathbb Q}$ (see Definition \ref{agr-12}). 
    \end{Thm}
\begin{proof} First, by Lemma \ref{agr-14}, every group $G\in \mathcal G$ can be isomorphically embedded into the symmetric group of $l_{m, N}$--letters. Next, there exists 
 $G\in  \mathcal G$ such that $\Xi_G$ is Zariski dense in $S_m(\Gamma_0(N))_{\mathbb Q}$ (see Lemma \ref{agr-15}). We fix such $G$. 

By Lemma \ref{agr-6},  there exists a non-empty Zariski open set $\calU\subset S_m(\Gamma_0(N))_{\mathbb Q}$ such that $\mathbb Q(g/f, h/f)=\mathbb Q(X_0(N))$ for $h\in \calU$.
Since $\Xi_G$ is Zariski dense in $S_m(\Gamma_0(N))_{\mathbb Q}$,  the intersection  $\Xi_G\cap \mathcal U$ is also Zariski dense in 
$S_m(\Gamma_0(N))_{\mathbb Q}$. Moreover, since  for $h\in \Xi_G\cap \mathcal U$, we have  $\mathbb Q(g/f, h/f)=\mathbb Q(X_0(N))$, we obtain that
the degree of $Q_{f, g, h}(g/f, \cdot)$ is $[\mathbb Q(X_0(N)): \mathbb Q(g/f)]$ (see Lemma \ref{agr-8}). Hence, $\Xi_G\cap \mathcal U\subset \Xi'_G$ by Definition \ref{agr-12}.
\end{proof}

\section{The Case of Hyperelliptic Modular Curves}\label{cmcg}
The goal of the present section is to prove  Theorem \ref{cmcg-1}. The reader should review Definitions \ref{agr-9} and \ref{agr-12}, and
Theorem \ref{agr-1}.

Recall from  \cite{Ogg} that  $X_0(N)$ is a  hyperelliptic curve if and only if   $N$ belongs to the set $\{22, 23, 26, 28, 29, 30,  31, 33, 35, 37, 39, 40, 41, 46, 47, 48, 50, 59, 71\}$.
 Select  $f, g\in S_2(\Gamma_0(N))_{\mathbb Q}$  such that their orders at $\mathfrak a_\infty$  (a cusp corresponding to $\infty$)  satisfy that
 $\nu_{\mathfrak a_\infty}(g)$ is largest possible, and  $\nu_{\mathfrak a_\infty}(f)= \nu_{\mathfrak a_\infty}(g) - 1$. The existence of $f$ and $g$ is easy to check using SAGE system. 
 
 \begin{Thm} \label{cmcg-1} Assume that $X_0(N)$ is a  hyperelliptic curve, and $f, g\in S_2(\Gamma_0(N))_{\mathbb Q}$  as above.
   Then, we have the following:
    \begin{itemize}
    \item[(i)]  The  extension $\mathbb Q(g/f)\subset \mathbb Q\left(X_0(N)\right)$  has the degree two, and therefore the Galois group is $\mathbb Z/2\mathbb Z$.
    \item[(ii)]  For all even integers $m\ge 4$ there exists  a non-empty Zariski open set
$\calU_m \subset S_m(\Gamma_0(N))_{\mathbb Q}$  such that
      $L_{f^{\frac{m}{2}}, gf^{\frac{m}{2}-1}, h }=\mathbb Q\left(X_0(N)\right)=  \mathbb Q\left(g/f, h/f^{m/2}\right)$, $h\in \calU_m$. We have  $\calU_m\subset \Xi'_{\mathbb Z/2\mathbb Z}$
      (see Definition \ref{agr-12}).
 Moreover, $\mathcal G_{f^{\frac{m}{2}}, gf^{\frac{m}{2}-1}}$  consists of $\mathbb Z/2\mathbb Z$ and the trivial group.
      \end{itemize}
  \end{Thm}
 \begin{proof} First, (i) follows from Corollary \ref{agr-16} and Lemma \ref{cmcg-2} (see Section \ref{tpoal}). Next, we prove that $\mathcal G_{f^{\frac{m}{2}}, gf^{\frac{m}{2}-1}}$ contains a trivial group. 
 Indeed, note that $f^{\frac{m}{2}}$, $gf^{\frac{m}{2}-1}$ and  $g^2f^{\frac{m}{2}-2}$ are $\mathbb Q$--linearly independent cuspidal modular
   forms in $S_m(\Gamma_0(N))_{\mathbb Q}$. We have
   $$
   \mathbb Q\left(gf^{\frac{m}{2}-1}/f^{\frac{m}{2}}, \ g^2f^{\frac{m}{2}-2}/f^{\frac{m}{2}}\right)= \mathbb Q\left(g/f, \ (g/f)^2\right)=
   \mathbb Q\left(g/f\right)= \mathbb Q\left(gf^{\frac{m}{2}-1}/f^{\frac{m}{2}}\right).
   $$
   This implies that  $\mathcal G_{f^{\frac{m}{2}}, gf^{\frac{m}{2}-1}}$  contains the trivial group. Next, by (i), 
    $$
   \mathbb Q\left(gf^{\frac{m}{2}-1}/f^{\frac{m}{2}}\right)=  \mathbb Q\left(g/f\right)\subset  \mathbb Q\left(X_0(N)\right)
   $$
   is a  degree two extension. Thus, all groups in $\mathcal G_{f^{\frac{m}{2}}, gf^{\frac{m}{2}-1}}$  are isomorphically contained in $S(2)=\mathbb Z/2\mathbb Z$. Now, 
   we prove the first  claim (ii) which also completes the proof of the last.  Since $g(\Gamma_0(N))\ge 2$,
for $m\ge 4$,  by the well--known dimension formula (see for example preliminary \cite[Lemma 2.2]{MuKo})  we have that 
  \begin{align*}
    \dim S_m(\Gamma_0(N))& \ge (m-1)(g(\Gamma_0(N))-1) + \\
    & +\left(\frac{m}{2}-1\right)\cdot \# \left(\text{the number of non--equivalent cusps for $\Gamma_0(N)$}\right)\\
    &\ge (m-1)(g(\Gamma_0(N))-1)  +\left(\frac{m}{2}-1\right) \ge g(\Gamma_0(N))+2.
  \end{align*}
  Thus,  by Lemma \ref{agr-6} (ii), there exists a    non-empty Zariski open set $\calU_m$ in  $S_m(\Gamma_0(N))_{\mathbb Q}$   such that
$\mathbb Q\left(X_0(N)\right)=  \mathbb Q\left(g/f, h/f^{m/2}\right)$, 
  for $h\in \calU_m$.   We obtain that
  the degree of $Q_{f^{\frac{m}{2}}, gf^{\frac{m}{2}-1}, h}(g/f, \cdot)$ is $[\mathbb Q(X_0(N)): \mathbb Q(g/f)]$ (see Lemma \ref{agr-8}). Hence, $\mathcal U_m\subset \Xi'_{\mathbb Z/2\mathbb Z}$
  by Definition \ref{agr-12}.
The claim (ii) follows. 
\end{proof}

We consider an example to Theorem \ref{cmcg-1}. Let $N=30$. Then, $g(\Gamma_0(30))=3$.  Using SAGE we find the following base of $S_2(\Gamma_0(30))$:
\begin{align*}
f_0 &=  q - q^4 - q^6 - 2q^7 + q^9 + q^{10}+\cdots\\
f_1 &=q^2 - q^4 - q^6 - q^8 + q^{10} + \cdots \\
f_2 &=q^3 + q^4 - q^5 - q^6 - 2q^7 - 2q^8 + q^{10}+\cdots \\
\end{align*}
We let $f=f_1$ and $g=f_2$. Now, we have that
\begin{align*}
  f^2 &= q^4 -2q^6 -q^8+ 5q^{12}+ \cdots\\
  fg &=   q^5 + q^6 - 2q^7 - 2q^8 - 2q^9 - 2q^{10} + 2q^{11}+ 3q^{12}\cdots
\end{align*}
are elements of $S_4(\Gamma_0(30))$. By listing the basis of $S_4(\Gamma_0(30))$ using SAGE, we construct a new base as follows: 
$F=F_0=f^2$, $G=F_1=fg$, $F_i=q^{i-1} +\ldots$, $2\le i \le 4$,  $F_i=q^{i+1} +\ldots$, $5\le i \le 14$.

Now, we use {\bf the Trial method} for $(F, G, H)$ (see \cite[Section 7]{MuKo}) where
$H=\sum_{i=2}^{13} a_i  F_i$,  $a_i\in \mathbb Z$, is to be determined. This method is very simple.  Given $H$, we compute the
polynom $P_{F, G, H}$ using SAGE. When
$$
\deg{P_{F, G, H}}> \left(\dim{S_{4}(\Gamma_0(30))} + g(\Gamma_0(30)) -1\right)/2= (14+ 3-1)/2=8,
$$
combining with Lemma \ref{agr-2}, we obtain
$\mathbb Q(X_0(N))=\mathbb Q\left(G/F, H/F\right)=\mathbb Q\left(g/f, H/f^2\right)$. 
Thus, in view of notation introduced in Theorem \ref{cmcg-1}, we may let
$h= H$.

Applying the algorithm, we may obtain $h=H=F_3$. The corresponding polynomial $Q_{f^2, fg, H}(\lambda, T)$ (a dehomogenization of $P_{f^2, fg, H}$
(see Definition \ref{agr-9})) is given by 
\begin{align*}
  &225 \lambda^6\left(1 -\lambda - \lambda^2 + \lambda^3\right)T^2 -\lambda^3 \big(237 - 370\lambda
+ 319 \lambda^2 + 341 \lambda^3 -310\lambda^4 -101 \lambda^5 + 400\lambda^6 -\\ & - 10\lambda^7 -
64\lambda^8 + 32\lambda^9 \big)T + 12 - 44 \lambda - 85 \lambda^2 + 153\lambda^3 + 1073\lambda^4
+ 1375\lambda^5 - 420\lambda^6 - 660\lambda^7 - \\ & - 30\lambda^8 + 
162 \lambda^9 -26\lambda^{10} - 118\lambda^{11} + 
84\lambda^{12} + 20\lambda^{13} + 12\lambda^{14} - 4\lambda^{15}.\\
\end{align*}

We see that $\mathbb Q(f/g)\subset \mathbb Q(X_0(N))$ is a quadratic exension as it should be by Theorem \ref{cmcg-1} (i).
Next, we apply methods of Section \ref{agr-hit} (see Theorem \ref{hit-main}). We easily see that $\wit{Q}_{f^2, fg, H}(\lambda, T)$ is given by
\begin{align*}
  &T^2 -\lambda^3 \big(237 - 370\lambda
+ 319 \lambda^2 + 341 \lambda^3 -310\lambda^4 -101 \lambda^5 + 400\lambda^6 - 10\lambda^7 -
64\lambda^8 + 32\lambda^9 \big)T  \\ &
+225 \lambda^6\left(1 -\lambda - \lambda^2 + \lambda^3\right)\cdot
\big(12 - 44 \lambda - 85 \lambda^2 + 153\lambda^3 + 1073\lambda^4
+ 1375\lambda^5 - 420\lambda^6 \\ & - 660\lambda^7 -  30\lambda^8 + 
162 \lambda^9 -26\lambda^{10} - 118\lambda^{11} + 
84\lambda^{12} + 20\lambda^{13} + 12\lambda^{14} - 4\lambda^{15}\big).\\
\end{align*}
We let $\lambda\in \mathbb Z- A_{f, g, h}$ and reduce polynomial  $(\text{mod} \ 5)$. We obtain 
$T^2 -\lambda^3 \big(2 -\lambda^2  + \lambda^3   - \lambda^5  +
\lambda^8 +2\lambda^9 \big)T$.  Letting $\lambda \equiv 1 \ (\text{mod} \ 5)$ (possible as explained in  the proof Theorem \ref{hit-main}) we obtain  $T^2-T=T(T-1)$. Considered $G_{f^2, fg, H, \lambda}$ as
a subgroup of the symmetric group $S(2)$, we see that it contains a
transposition. Hence, $G_{f^2, fg, H, \lambda}=S(2)$. This recovers the Galois group of the  original polynomial $G_{f^2, fg, H}$ by
using Hilbert's irreducibility. More interesting examples are given in Section \ref{cnmc}.

\section{The Case of Non--Hyperelliptic Modular Curves}\label{cnmc}

By \cite{Ogg},  $X_0(N)$ is non--hyperelliptic for $N\in \{34, 38, 42, 43, 44, 45, 51-58, 60-70\}$ or $N\ge 72$. This implies $g(\Gamma_0(N))\ge 3$.
We assume that in this section. The main result of the present section is the following theorem. The reader should compare to Theorem \ref{cmcg-1}. The bound for 
$[\mathbb Q(X_0(N)): \mathbb Q(g/f)]$ in the theorem below is far more better than the generic bound $[\mathbb Q(X_0(N)): \mathbb Q(g/f)]\le 2g(\Gamma_0(N))-2$ given by Corollary \ref{agr-16}. 

\begin{Thm}\label{sec-1} Maintaining above assumptions,  we select $f$ and $g$ in $S_m(\Gamma_0(N))_{\mathbb Q}$ with largest possible orders of vanishing at $\mathfrak a_\infty$ (as in Theorem \ref{cmcg-1}
  but they are not necessarily consecutive here), $\nu_\infty(f)< \nu_\infty(g)$. Then, $[\mathbb Q(X_0(N)): \mathbb Q(g/f)] \le g(\Gamma_0(N))$. Consequently, for
  $h\in S_2(\Gamma_0(N))_{\mathbb Q} - \left(\mathbb Q f+ \mathbb Q g\right)$, $G_{f, g, h}$ can be embedded as  a subgroup of the symmetric group of 
  $g(\Gamma_0(N))$--letters $S_{g(\Gamma_0(N))}$ (non--uniquely). Moreover, there exists a subgroup $G$ of $S_{g(\Gamma_0(N))}$ such that $\Xi'_G$ is Zariski dense in (see Definition \ref{agr-12}).
  \end{Thm}
\begin{proof}  We let  $W_i=\left\{f\in S_2(\Gamma_0(N)); \  f=0 \ \text{or} \ \nu_{\mathfrak a_\infty}(f)\ge i\right\}$. Then, by the proof of
  \cite[Corollary 5.6]{MuKo},  $\dim W_{g(\Gamma_0(N))-2}\ge 3$. This implies  $\dim   W_{g(\Gamma_0(N))-1}\ge 2$. Obviously, we have $f, g\in
  W_{g(\Gamma_0(N))-1}$. Now, in (\ref{agr-10}) there are a lot more cancellations. Namely, $q$--expansions satisfy 
  $$
  \begin{cases}
    f &= a_lq^l + \cdots, \ \ a_l\neq 0, \ \ l=\nu_\infty(f),\\
     g&= b_{k}q^{k} + \cdots, \ \ b_{k}\neq 0, \ \ k=\nu_\infty(g),\\
  \end{cases}
  $$
  where $k> l\ge g(\Gamma_0(N))-1$.   This implies that
  $$\begin{cases}
    \mathfrak c_f &= (l-1)\mathfrak a_\infty + \cdots, \\
     \mathfrak c_g &= (k-1)\mathfrak a_\infty + \cdots, \\
    \end{cases}$$
  by definition of these divisors \cite[Lemma 2.1]{Muic2}. Thus,   by (\ref{agr-10}), the divisor of zeroes of $g/f$ has degree
  $$
  2(g(\Gamma_0(N))-1)- (l-1)\le  2(g(\Gamma_0(N))-1) - (g(\Gamma_0(N))-2)= g(\Gamma_0(N)).
  $$
  By Lemma \ref{agr-6} and Corollary \ref{agr-16}, we obtain $[\mathbb Q(X_0(N)): \mathbb Q(g/f)] \le g(\Gamma_0(N))$.  The rest of claims follow from Lemma \ref{agr-13} and
  Theorem \ref{agr-1}.
\end{proof}

We continue with an example related to \cite[Corollary 5.8]{MuKo} which shows that the result in Theorem \ref{sec-1} is optimal.
Consider three basis elements of $5$--dimensional space   $S_2(\Gamma_0(63))$ having highest order of zero at $\mathfrak a_\infty$:
	\begin{align*}  
	f\overset{def}{=}&q^4 + q^7 - 4q^{10} + 2q^{13} - 2q^{16} - 4q^{19} + 5q^{22} + \cdots,\\
	g\overset{def}{=}& 2q^5 - q^8 - 3q^{11} - q^{14} + 2q^{17} + q^{23} +\cdots, \\
        h\overset{def}{=}& q^3 - q^6 + q^9 - q^{12} - 2q^{15} - q^{18} - q^{21} + 3q^{24} + \cdots.\\
	\end{align*}
        The reader should observe that $f$ and $g$ are chosen as in the hyperelliptic case (see Theorem \ref{cmcg-1}).
        We apply methods of Section \ref{agr-hit} (see Theorem \ref{hit-main}).
        
        \begin{Prop} \label{sec-3} Maintaining above assumptions, we have $G_{f, g, h}\simeq S(5)$. Moreover, $h\in \Xi'_{S(5)}$, and  $\mathbb Q(g/f, h/f)=\mathbb Q(X_0(63))$. 
  \end{Prop}
\begin{proof}   The reduced equation of the curve $\calC (f, g, h)$  is computed in SAGE. It is given by the polynomial $P_{f, g, h}$  irreducible over $\mathbb Q$  
  $$
  -2 h^{4} f^{2} -  h f^{5} + h^{5} g + 2 h^{2} f^{3} g + h^{3} f g^{2} -  f^{4} g^{2} + 3 h f^{2} g^{3} - 3 h^{2} g^{4}=P_{f, g, h}(f, g, h)=0.
  $$
This implies 
  $$
 Q_{f, g, h}(\lambda, T)= \lambda T^{5} -2 T^{4} +   \lambda^{2} T^3+ \left(2  \lambda - 3 \lambda^4 \right)T^{2}   + \left(3   \lambda^{3} -1\right)T -   \lambda^{2}.
  $$
 We have (see Section \ref{agr-hit})
 \begin{equation} \label{sec-3000}
   \wit{Q}_{f, g, h}(\lambda, T) = T^{5} -2 T^{4} +   \lambda^{3} T^3+ \left(2  \lambda - 3 \lambda^4 \right)\lambda^2 T^{2}   + \left(3   \lambda^{3} -1\right)\lambda^3 T -   \lambda^{6}.
  \end{equation}
For $\lambda \equiv -1 \ (\text{mod} \ 3)$, reducing 
  $\equiv \ (\text{mod} \ 3)$, the polynomial (\ref{sec-3000}) becomes $T^5+T^4-T^3+T^2+T+1$ which is irreducible over $\mathbb Z/3\mathbb Z$.  Since $5$ is a prime number,
we see that that group $G_{f, g, h}$ must contain a $5$--cycle (see Theorem \ref{hit-main}). 
Similarly, for $\lambda \equiv -1 \ (\text{mod} \ 7)$, the polynomial  (\ref{sec-3000}) becomes a product of two irreducible polynomials
  $$
  T^5-2T^4 -T^3 +2T^2+3T-1= \left(T^2-T+3\right) \cdot \left(T^3-T^2+4T+2 \right).
  $$
  This shows that the Galois group $G_{f, g, h}$ contains a permutation which is a product of commuting $2$--cycle and $3$--cycle. Its cube is a transposition.
  Now, since $G_{f, g, h}$ is a subgroup of $S(5)$ (and $5$ is a prime number) which contains a   $5$--cycle and a transposition,  by the standard argument, the group is $S(5)$
  (see \cite{lang}). Next, by Definition \ref{agr-12}, to prove  $h\in \Xi'_{S(5)}$, we need to show that $[\mathbb Q(X_0(N)): \mathbb Q(g/f)]=g(\Gamma_0(63))=5$. Indeed, by 
  above considerations and Lemma \ref{agr-8}, we have that 
  $$
  [\mathbb Q(g/f, h/f): \mathbb Q(g/f)]= \deg{ Q_{f, g, h}(g/f, \cdot)}=5.
  $$
  But, since by Theorem \ref{sec-1} we have
  $$
  [\mathbb Q(X_0(63)): \mathbb Q(g/f)]\le g(\Gamma_0(63))=5,
  $$
  we obtain
  $\mathbb Q(g/f, h/f)=\mathbb Q(X_0(63))$. \footnote{
  This also follows by methods of \cite{Muic2} from  \cite[Corollary 5.8]{MuKo} using Lemma \ref{agr-2}.} The claim follows. Finally, the last claim of the proposition  is also proved. 
\end{proof}

\begin{Rem} Keeping $f$ and $g$ fixed and letting $h$ we have computed a considerable number of examples using MAGMA routines mentioned in the introduction. We always obtained $G_{f, g, h}\simeq S(5)$.
  \end{Rem}

Let $N=  64$. Then, $g(\Gamma_0(64))= 3$. The basis of $S_2(\Gamma_0(63))$ is 
\begin{align*}
f&\overset{def}{=} q^2 - 2q^{10} - 3q^{18} + 6q^{26} + 2q^{34} +\cdots, \\
g&\overset{def}{=} q^5 - 3q^{13} + 5q^{29} + q^{37} + \cdots, \\
h&\overset{def}{=}q - 3q^9 + 2q^{17} - q^{25} + \cdots, \\
\end{align*}
Using SAGE we check that $-f^{4} + h^{3} g + 4 h g^{3}=0$. This implies that $Q_{f, g, h}(\lambda, T)= \lambda T^3 + 4\lambda^3 T -1$. Arguing as in Proposition \ref{sec-3} but reducing
$(\text{mod} \ 2)$   and $(\text{mod} \ 5)$ we easily establish. 

\begin{Prop} \label{sec-4} Maintaining above assumptions, we have $G_{f, g, h}\simeq S(3)$ and $\mathbb Q(g/f, h/f)=\mathbb Q(X_0(64))$. 
  \end{Prop}

The last set of examples is for $N=72$. We have  $g(\Gamma_0(72))=5$. Using SAGE, the basis  of $5$--dimensional space
$S_2(\Gamma_0(72))$ is given by 
\begin{align*}
f\overset{def}{=}f_0 &\overset{def}{=}q^5 - 2q^{11} - q^{17} + 4q^{23} - 3q^{29} + \cdots,\\
g\overset{def}{=}f_1&\overset{def}{=}q^7 - q^{13} - 3q^{19} + q^{25} + 3q^{31} + 4q^{37} +  \cdots, \\
f_2&\overset{def}{=} q^3 - q^9 - 2q^{15} + q^{27} + 4q^{33} - 2q^{39} + \cdots,\\
f_3 &\overset{def}{=} q^2 - 4q^{14} + 2q^{26} + 8q^{38} + \cdots,\\
f_4&\overset{def}{=} q - 2q^{13} - 4q^{19} - q^{25} + 8q^{31} + 6q^{37} + \cdots.
\end{align*}

In the following proposition, among other things, we  show that  $[\mathbb Q(X_0(72)): \mathbb Q(g/f)]=4$ (which is $< g(\Gamma_0(72))=5$, see Theorem \ref{sec-1}).

\begin{Prop} \label{sec-5} Let  $h=f_3$. Then, we have that  $G_{f, g, h}\simeq D(4)$
  a dihedral group of order $2\cdot 4=8$. Moreover, $\mathbb Q(g/f, h/f)=\mathbb Q(X_0(72))$,  $[\mathbb Q(X_0(72)): \mathbb Q(g/f)]=4$, and $h\in \Xi'_{D(4)}$. Moreover,
  the Galois group of the extension $\mathbb Q(X_0(72)) \subset L_{f, g, h}$ is generated by a transposition in $D(4)$. 
\end{Prop}
\begin{proof} The polynomial $P_{f, g, h}$ is given by
  $$
  f^7 + 7f^4g^3 - 8fg^6 - f^5h^2 - 8f^2g^3h^2 + g^3h^4.
  $$
  This implies
  $$
  \wit{Q}_{f, g, h}(\lambda, T)= T^4 - \lambda^3 (1+8\lambda^3) T^2 + \lambda^9 (1+7\lambda^3 -8\lambda^6).
  $$
  For $\lambda \equiv -1 \ (\text{mod} \ 3)$,  reducing 
  $\equiv \ (\text{mod} \ 3)$, the polynomial $\wit{Q}_{f, g, h}(\lambda, T)$ becomes
  $T^4-T^2-1$ which is irreducible over $\mathbb Z/3\mathbb Z$. Thus, $G_{f, g, h}$ must contain a $4$--cycle (see Theorem \ref{hit-main}). 
   Next, for $\lambda \equiv -1 \ (\text{mod} \ 3)$,  reducing 
  $\equiv \ (\text{mod} \ 5)$, the polynomial $\wit{Q}_{f, g, h}(\lambda, T)$ becomes
   $T^4-2T^2-1$ which is product of irreducible polynomials $T^2-T+2$ and $T^2+T+2$ over  $\mathbb Z/5\mathbb Z$. This shows that the Galois group $G_{f, g, h}$ contains a permutation with a cycle pattern 2, 2.
   Now, we list all transitive groups in $S(4)$ (see http://galoisdb.math.upb.de/groups?deg=4  and \cite{conhulmckay} for notation). They are cyclic group $C(4)$, the Klein group $E(4)$ (order $4$ all
   non--trivial elements have order 2), the dihedral group $D(4)$ (order $8$), the alternating group $A(4)$, and the symmetric group $S(4)$. Clearly, $G_{f, g, h}$  cannot be $C(4)$. It is not also $E(4)$ since
   it contains a four cycle which has order $4$. Also, a four cycle is a product of three transpositions, for example $(1, 2, 3, 4)=(1, 2)(1, 3) (1, 4)$. Thus, $G_{f, g, h}$ cannot be  $A(4)$. Thus, $G_{f, g, h}$
   is either $D(4)$ or $S(4)$. If it would be $S(4)$, the it would contain a permutation with a cycle pattern 1, 3. Then, by Theorem \ref{hit-main}, there would exist a prime number $p > 2$ and $\lambda
   \in \mathbb Z$
   such that   reducing 
  $\equiv \ (\text{mod} \ p)$ the polynomial  $\wit{Q}_{f, g, h}(\lambda, T)$ would be a product of a normalized linear polynomial, say $P_1$, and a cubic normalized irreducible polynomial, say $P_2$,
   over $\mathbb Z/p\mathbb Z$. Clearly, $P_1$  is of the form $T-a$, $a\not\equiv  0 \ (\text{mod} \ p)$. Then, $a$ is a zero of $\wit{Q}_{f, g, h}(\lambda, T)$ modulo $p$. Clearly, $-a$ is also a zero, and
   since $p>2$ we have $a\not\equiv -a \ (\text{mod} \ p)$.
   This is a contradiction since $P_2$ cannot have zeroes in $\mathbb Z/p\mathbb Z$.

   After renumeration, the group $D(4)$ is generated by a $4$--cycle $r=(1 2 3 4)$ and a transposition $s=(1 3)$ subject to the relation $srs=r^{-1}$. The elements of $D(4)$
   are $1, r, r^2, r^3, r^4$, $sr=(2 3)(1 4)$,
   $sr^2=(2 4)$ and $sr^3= (1 2) (3 4)$. Thus, in the chain of field extensions $\mathbb Q(g/f)\subset \mathbb Q(g/f, h/f)\subset L_{f, g, h}$, we see that
   $Gal(L_{f, g, h}/ Q(g/f, h/f))$ must be generated either by $s=(1 3)$ or $sr^2=(2, 4)$  since other permutations do not have fixed points and a root $h/f$ of $Q_{f, g, h}(g/f, T)$ must be fixed by elements of 
   $Gal(L_{f, g, h}/ Q(g/f, h/f))$.

   Finally, since the $\deg{P_{f, g, h}}=7> g(\Gamma_0(72))-1=4$, the Trial method \cite[Section 7]{MuKo} and Lemma \ref{agr-2} imply 
   $\mathbb Q(g/f, h/f)=\mathbb Q(X_0(72))$. Hence, $[\mathbb Q(X_0(72)): \mathbb Q(g/f)]=4$.  In particular,  by Definition \ref{agr-12},  $h\in \Xi'_{D(4)}$.
\end{proof}

\begin{Rem} By the method described in Section \ref{agr-hit} (see Theorem \ref{hit-main}), a transposition in $G_{f, g, h}=D(4)$ can be seen by taking  $\lambda \equiv -1 \ (\text{mod} \ 79)$
  and reducing $ \wit{Q}_{f, g, h}(\lambda, T)$ modulo $p=79$. We obtain the polynomial $T^4-7T^2 +14$ which is a product of two linear polynomials $T\pm 4$ and an irreducible quadratic polynomial
  $T^2+9$. By Theorem \ref{hit-main}, $G_{f, g, h}$ must contain a permutation with a cyclic pattern $ 1, 1,  2$ i.e., a transposition.
  \end{Rem}

 \begin{Prop} \label{sec-6} Let $h=f_3+f_4$. Then, we have $G_{f, g, h}\simeq S(4)$. Moreover, $h\in \Xi'_{S(4)}$, and  $\mathbb Q(g/f, h/f)=\mathbb Q(X_0(72))$. 
  \end{Prop}
 \begin{proof} The proof is very similar to the proof of Proposition \ref{sec-5}, but easier. So, we sketch the argument. First, using SAGE, we find that
   $P_{f, g, h}$ is given by
\begin{align*}
& f^8 - f^7g + 8f^5g^3 - 7f^4g^4 + 8fg^7 - 16g^8 -
2f^5g^2h + 8f^3g^4h - 16f^2g^5h + \\&32g^7h - f^6h^2 +
f^5gh^2 - 8f^3g^3h^2 +
8f^2g^4h^2 - 24g^6h^2 +
2f^3g^2h^3 + 8g^5h^3 - g^4h^4. 
\end{align*}
Hence, $\wit{Q}_{f, g, h}(\lambda, T)$ is given by 
\begin{align*}
& -\lambda^{12}( 1 - \lambda + 8\lambda^3 - 7\lambda^4 + 8\lambda^7 - 16\lambda^8)  + \lambda^8
  (2\lambda^2 + 8\lambda^4 - 16\lambda^5 + 32\lambda^7)T -\\&
 - \lambda^4 (- 1 + \lambda  - 8\lambda^3 +   8\lambda^4 - 24\lambda^6)T^2 +
  (2\lambda^2 + 8\lambda^5) T^3 + T^4
\end{align*}

For $\lambda \equiv 1 \ (\text{mod} \ 3)$,  reducing 
$\equiv \ (\text{mod} \ 3)$, the polynomial $\wit{Q}_{f, g, h}(\lambda, T)$ becomes $1-T+ T^3+ T^4$. This polynomial is irreducible
over $\mathbb Z/3\mathbb Z$. Hence, by Theorem \ref{hit-main}, the group $G_{f, g, h}$ contains a four cycle.
Similarly for $\lambda \equiv 1 \ (\text{mod} \ 7)$,  reducing 
$\equiv \ (\text{mod} \ 7)$, the polynomial $\wit{Q}_{f, g, h}(\lambda, T)$ becomes $(-2 + 3T + 3T^2  +T^3)\cdot T$ a product of two irreducible polynomials.
Hence, $G_{f, g, h}$ contains a permutation with a cycle patter $1, 3$. Now, as in the proof of Proposition \ref{sec-5},
in view of  list all transitive groups in $S(4)$, we see that $G_{f, g, h}$ must be $S(4)$. The other claims follow by the same arguments as the analogous claims in
Proposition \ref{sec-5}. 
\end{proof}

 \begin{Prop} \label{sec-7} Maintaining above assumptions, let $h=f_2$. Then, we have $G_{f, g, h}\simeq \mathbb Z/ 2\mathbb Z$. Next, $h\in \Xi_{ \mathbb Z/ 2\mathbb Z}$ but
   $\Xi'_{ \mathbb Z/ 2\mathbb Z}=\emptyset$. 
\end{Prop}
 \begin{proof} As before, using SAGE, we find  that $-f^2h+ gh^2-2fg^2=0$. Hence, we see $Q_{f, g, h}(\lambda, T)= \lambda T^2 - T -2\lambda^2$. Hence,
   $\wit{Q}_{f, g, h}(\lambda, T)=T^2 - T -2\lambda^3$. 
  The result easily follows. For example, for $\lambda \equiv 0 \ (\text{mod} \ 2)$,  reducing 
  $\equiv \ (\text{mod} \ 2)$, the polynomial $\wit{Q}_{f, g, h}(\lambda, T)$ becomes a product $T\cdot (T-1)$. Hence, $G_{f, g, h}\simeq S(2) \simeq \mathbb Z/ 2\mathbb Z$. In particular,
  $h\in \Xi_{ \mathbb Z/ 2\mathbb Z}$.    Finally, for $h'\in \Xi_{ \mathbb Z/ 2\mathbb Z}$, the degree of the extension $\mathbb Q(g/f)\subset \mathbb Q(g/f, h'/f)$ is two while
  $\mathbb Q(g/f)\subset \mathbb Q(X_0(72))$ has degree four by Proposition \ref{sec-5}.
  Hence, $\Xi'_{ \mathbb Z/ 2\mathbb Z}=\emptyset$ (see Definition \ref{agr-12}).
  \end{proof}

 \begin{Rem} \label{sec-8} Using  http://galoisdb.math.upb.de/groups?deg=4  (see \cite{conhulmckay} for notation), we may conclude  that $\Xi_{C(4)}=\Xi_{E(4)}=\emptyset$ (see Definition \ref{agr-12}). Indeed,
   let $h'\in  \Xi_{C(4)}$ or $h'\in  \Xi_{E(4)}$. Then,
   the degree $[\mathbb Q(g/f, h'/f): \mathbb Q(f/g)]$ obviously must divide  $[\mathbb Q(X_0(72)): \mathbb Q(f/g)]=4$ (see Proposition \ref{sec-5}). Thus, $[\mathbb Q(g/f, h'/f): \mathbb Q(f/g)]=4$
   (or otherwise  $L_{f, g,h'}= \mathbb Q(g/f, h'/f)$ contradicting that it must be a degree four extension of $\mathbb Q(f/g)$).  We obtain that $\mathbb Q(g/f, h'/f)=\mathbb Q(X_0(72))$ is a Galois extension of
   $\mathbb Q(f/g)$ with the Galois group
   $C(4)$. On the other hand, using $h$ from Proposition \ref{sec-6} we have the Galois extension $L_{f, g, h}$ of  $\mathbb Q(f/g)$ with the Galois group
   $S(4)$  containing $\mathbb Q(X_0(72))$. Now, by Galois theory we conclude that $S(4)$ has normal group with a quotient of order four. But this is a contradiction since $S(4)$ has
   no normal subgroup of order $4!/4=6$. 
 \end{Rem}

\section{The Proof of Lemma \ref{cmcg-2}} \label{tpoal}
In this section we prove the following lemma used in the proof of Theorem \ref{cmcg-1}:

\begin{Lem} \label{cmcg-2} Maintaining the assumptions of Theorem \ref{cmcg-1}, we have
  that $\mathbb C\left(g/f\right) \subset \mathbb C\left(X_0(N)\right)$ is a degree two extension.
\end{Lem}
\begin{proof} This claim is  contained in \cite[Lemma 2]{mshi} but we give  different  proof
using our previous work \cite{Muic2} combined with certain simple computations in SAGE.

First, we assume that we have $g(\Gamma_0(N))=2$ i.e.,
$N=22, 23, 26, 28, 29, 31, 37, 50$. 
We use results recalled in the proof of Theorem \ref{agr-1} up to the formula in (\ref{agr-5}).
The divisor of zeroes must be $\mathfrak c_g$ since otherwise the divisor of zeroes would have degree one, and we would obtain
$\mathbb C(X_0(N))=\mathbb C(g/f)$. Hence,  $X_0(N)$ would have a  genus $0$ which is a
contradiction. This proves  (\ref{cmcg-2}) when  $g(\Gamma_0(N))=2$. Next, we assume  that $g(\Gamma_0(N))\ge 3$ i.e.,
\begin{equation}\label{cmcg-4-N}
  N=30,  33, 35,  39, 40, 41, 46, 47, 48,  59, 71.
\end{equation}
Then, above method does not work since the degree of divisors $\mathfrak c_f$ and $\mathfrak c_g$ is $2(g(\Gamma_0)N))-1)\ge 4$.
We need to dig deeper into the geometry and computations in SAGE. We write $f_0, \ldots, f_{g(\Gamma_0(N))-1}$ for the basis of $S_2(\Gamma_0(N))$. By Eichler--Shimura theory
\cite[Theorem 3.5.2]{shi} we see that they have integral $q$--expansions, and using SAGE system we can easily check that their $q$--expansions are of the form
\begin{equation}\label{cmcg-4-q}
f_i=a_iq^{i+1}+ \cdots,  \ \ 0\le i\le g(\Gamma_0(N))-1.
\end{equation}
where $a_i\neq 0$. By definition of $f$ and $g$, we see that $g=\lambda f_{g(\Gamma_0(N))-1}$ and $f=\mu f_{g(\Gamma_0(N))-2}+ \nu f_{g(\Gamma_0(N))-1}$, 
for some $\lambda, \mu, \nu \in\mathbb Q$, $\lambda, \mu \neq 0$. So, in  order to prove (\ref{cmcg-2}), we may assume $f=f_{g(\Gamma_0(N))-2}$, and $g=f_{g(\Gamma_0(N))-1}$.

In what follows we do not need letter $g$ to denote a cupidal form, so we may put $g=g(\Gamma_0(N))$. We consider the analogue of the map (\ref{map}) given by
\begin{equation}\label{cmcg-3}
\mathfrak{a}_z\longmapsto (f_{g-3}(z): f_{g-2}(z): f_{g-1}(z)), \ \ z\in \mathbb H.
\end{equation}
Let $d(f_{g-3}, f_{g-2},  f_{g-1})$ be the degree of that map (it is explained in detail in \cite{Muic2}). We recall that $d(f_{g-3}, f_{g-2},  f_{g-1})$  is a degree of the field extension
$\mathbb C\left(\calC(f_{g-3}, f_{g-2}, f_{g-1})\right)\subset \mathbb C\left(X_0(N)\right)$. 
This means that the map (\ref{cmcg-3}) is generically $d$--to--$1$ map (see \cite[Lemma 3.4]{Muic2}). We prove
\begin{equation}\label{cmcg-4-even}
\text{$d(f_{g-3}, f_{g-2},  f_{g-1})$ is an even integer.}
\end{equation}
In order to prove (\ref{cmcg-4-even}),  let $K$ be the canonical divisor on $X_0(N)$  i.e., a divisor attached to a non--zero meromorphic differential form on $X_0(N)$.
The regular map attached to $K$ is is explicitly described in \cite{sgal}: 
\begin{equation}\label{cmcg-4}
X_0(N)\longrightarrow \mathbb P^{g-1}, \ \ \mathfrak a_z \longmapsto \left(f_0(z): f_1(z): \cdots : f_{g-1}\right), \ \ z\in \mathbb H.
\end{equation}
 The image is an irreducible projective curve, say $\calC(f_0, \ldots, f_{g-1})$. It is well--known that this maps is generically two--to--one since our curve is hyperelliptic.
This means that the degree of the field extension $\mathbb C\left(\calC(f_0, \ldots, f_{g-1})\right) \subset \mathbb C\left(X_0(N)\right)$
is equal to $2$. 

The map (\ref{cmcg-3}) can be considered as a rational map (i.e., defined on a open subset), and factored as a dominant rational map in  a composition of 
the map (\ref{cmcg-4}) followed by the restriction of the projection  $\mathbb P^{g-1}\longrightarrow \mathbb P^2$ given by 
$(y_0: y_1:\cdots: y_{g-1})\longrightarrow (y_{g-3}: y_{g-2}: y_{g-1})$
which is defined on a non--empty Zariski open set of the image $\calC(f_0, \ldots, f_{g-1})$. (For example, on the complement of the (finite) set of zeroes in $\mathbb H$ of the cuspidal modular form
$f_{g-3}f_{g-2}f_{g-1}$ of weight $6$.) Let $d$ be the degree of the projection restricted to the curve $\calC(f_0, \ldots, f_{g-1})$. By definition, this is a degree of the field
extension $\mathbb C\left(\calC(f_{g-3}, f_{g-2}, f_{g-1})\right) \subset \mathbb C\left(\calC(f_0, \ldots, f_{g-1})\right)$.
Thus, considering chain field extensions
$\mathbb C\left(\calC(f_{g-3}, f_{g-2}, f_{g-1})\right) \subset \mathbb C\left(\calC(f_0, \ldots, f_{g-1})\right)\subset 
 \mathbb C\left(X_0(N)\right)$, 
we see that $d(f_{g-3}, f_{g-2},  f_{g-1})=2d$. This proves (\ref{cmcg-4-even}).

Next, by the main result of \cite{Muic2}, the product $d(f_{g-3}, f_{g-2}, f_{g-1}) \cdot \deg{P_{f_{g-3}, f_{g-2}, f_{g-1}}}$ is given by 
  $$
  2(g -1) - \sum_{\mathfrak a\in X_0(N)} 
\min{\left(\mathfrak c_{f_{g-3}}(\mathfrak a),  \mathfrak c_{f_{g-2}}(\mathfrak a), 
  \mathfrak c_{f_{g-1}}(\mathfrak a) \right)},
$$
where effective divisors $\mathfrak c_{f_{g-3}}, \mathfrak c_{f_{g-2}}$, and $\mathfrak c_{f_{g-1}}$ are defined by \cite[Lemma 2.1]{Muic2}.
By (\ref{cmcg-4-q}), we have
$$
\sum_{\mathfrak a\in X_0(N)} 
\min{\left(\mathfrak c_{f_{g-3}}(\mathfrak a),  \mathfrak c_{f_{g-2}}(\mathfrak a), 
  \mathfrak c_{f_{g-1}}(\mathfrak a) \right)}\ge \min{\left(\mathfrak c_{f_{g-3}}(\mathfrak a_\infty),  \mathfrak c_{f_{g-2}}(\mathfrak a_\infty), 
  \mathfrak c_{f_{g-1}}(\mathfrak a_\infty) \right)}=g-3.
$$

Thus, we obtain the following:
\begin{equation} \label{cmcg-5}
d(f_{g-3}, f_{g-2}, f_{g-1}) \cdot \deg{P_{f_{g-3}, f_{g-2}, f_{g-1}}} \le 2(g-1)-(g-3)= g+1.
\end{equation}

Considering case by case analysis based on (\ref{cmcg-4-N}) (described below), we  compute equations in SAGE, and   find  that $\deg{P_{f_{g-3}, f_{g-2}, f_{g-1}}}=2$
for all $N$ in  (\ref{cmcg-4-N}). Using (\ref{cmcg-4-even}) and (\ref{cmcg-5}), we obtain  $d(f_{g-3}, f_{g-2}, f_{g-1})=2$ in all cases. This means that
$\mathbb C\left(X_0(N)\right)$ is a quadratic extension of $\mathbb C\left(\calC(f_{g-3}, f_{g-2}, f_{g-1})\right)$.  Since the explicit computations below show
$\mathbb C\left(\calC(f_{g-3}, f_{g-2}, f_{g-1})\right)$ is equal to  $\mathbb C\left(f_{g-1}/f_{g-2}\right)$, the proof lemma is complete.

Now, we give case by case analysis mentioned above. Assume $N=30$. We have $g=3$. The equation of $\mathcal C(f_{g-3}, f_{g-2}, f_{g-1})$ is
given by the irreducible polynomial $-x^2_1+ x_0x_2-x_1x_2$. We have
$$
\mathbb C\left(\calC(f_{g-3}, f_{g-2}, f_{g-1})\right)=\mathbb C\left(f_{g-3}/f_{g-2}, f_{g-1}/f_{g-2} \right)=  \mathbb C\left(f_{g-1}/f_{g-2}\right)
  $$
since the equation of the curve implies $-f^2_{g-2}+f_{g-3}f_{g-1} - f_{g-2}f_{g-1}=0$.  Analogously, we treat remaining cases.  We leave  details to the
reader. 

Assume that  $N=33, 35, 39, 40, 41$, or $48$. We have $g=3$. The equations of curves $\mathcal C(f_{g-3}, f_{g-2}, f_{g-1})$ are given by the irreducible polynomials
$-x^2_1+ x_0x_2- 2x^2_2$,  $-x^2_1-x_0x_2+2x_1x_2-x^2_2$, $-x^2_1 +x_0x_2 +x_1x_2$, $-x^2_1+x_0x_2+x^2_2$, $-x_1^2+x_0x_2+2x_1x_2$, and $-x^2_1+x_0x_2$, respectively.

Assume  $N=47$. Then, $g=4$.  The equation of $\mathcal C(f_{g-3}, f_{g-2}, f_{g-1})$ is given by the irreducible polynomial $-x^2_1+x_0x_2+ 2x_1x_2-3x^2_2$.   

Assume that  $N=46$ or $59$. Then, $g=5$.  The equations of $\mathcal C(f_{g-3}, f_{g-2}, f_{g-1})$ are  given by the irreducible polynomials $-x^2_1+ x_0x_2-x_1x_2$ and $-x^2_1+x_0x_2+ 2x_1x_2-3x^2_2$,
respectively.

Assume $N=71$. Then, $g=6$.  The equation of $\mathcal C(f_{g-3}, f_{g-2}, f_{g-1})$ is given by the irreducible polynomial $-x^2_1+x_0x_2+ 2x_1x_2-3x^2_2$.
\end{proof}

\end{document}